\documentclass[reqno]{amsart}

\usepackage{amsfonts}
\usepackage{amssymb}
\usepackage{amsmath}
\usepackage{amsthm}
\usepackage{bbm}
\usepackage{graphics}
\usepackage{mathrsfs}
\usepackage{float}

\newcommand{\zz}{\mathbb{Z}}
\newcommand{\nn}{\mathbb{N}}

\newcommand{\gf}{\mathbb{F}}

\newcommand{\ceil}[1]{\left\lceil#1\right\rceil}

\newcommand{\ord}{\mathrm{ord}}
\newcommand{\N}{\mathrm{N}}

\newtheorem{prop}{Proposition}
\newtheorem{lma}{Lemma}
\newtheorem{thm}{Theorem}

\newtheorem{fact}{Fact}

\usepackage{fullpage}
\begin{document}
\title{Finite field elements of high order arising from modular curves\\
(Appeared in \textit{Designs, Codes, and Cryptography})}

\bibliographystyle{plain}
\author{Jessica F. Burkhart}
\address{
Jessica F. Burkhart\\
Department of Mathematical Sciences\\
Clemson University\\
Box 340975 Clemson, SC 29634-0975
}
\email{burkhar@clemson.edu}

\author{Neil J. Calkin}
\address{
Neil J. Calkin\\
Department of Mathematical Sciences\\
Clemson University\\
Box 340975 Clemson, SC 29634-0975
}
\email{calkin@clemson.edu}

\author{Shuhong Gao}
\address{
Shuhong Gao\\
Department of Mathematical Sciences\\
Clemson University\\
Box 340975 Clemson, SC 29634-0975}
\email{sgao@clemson.edu}

\author{Justine C. Hyde-Volpe}
\address{
Justine C. Hyde-Volpe\\
Department of Mathematical Sciences\\
Clemson University\\
Box 340975 Clemson, SC 29634-0975
}
\email{jchasma@clemson.edu}

\author{Kevin James}
\address{
Kevin James\\
Department of Mathematical Sciences\\
Clemson University\\
Box 340975 Clemson, SC 29634-0975
}
\email{kevja@clemson.edu}

\author{Hiren Maharaj}
\address{
Hiren Maharaj\\
Department of Mathematical Sciences\\
Clemson University\\
Box 340975 Clemson, SC 29634-0975}
\email{hmahara@clemson.edu}

\author{Shelly Manber}
\address{
Shelly Manber\\
Department of Mathematics\\
Massachusetts Institute of Technology\\
Cambridge, MA 02139}
\email{shellym@mit.edu}

\author{Jared Ruiz}
\address{
Jared Ruiz\\
Department of Mathematics and Statistics\\
Youngstown State University\\
One University Plaza\\
Youngstown, OH 44555}
\email{jmruiz@student.ysu.edu}

\author{Ethan Smith}
\address{
Ethan Smith\\
Department of Mathematical Sciences\\
Clemson University\\
Box 340975 Clemson, SC 29634-0975
}
\email{ethans@math.clemson.edu}
\thanks{Burkhart, Calkin, Hyde-Volpe, James, Manber,  Ruiz,  and  Smith
were partially supported by the NSF grant DMS 0552799, 
and Gao was partially supported by NSF grant DMS 0302549.}

\begin{abstract}
In this paper, we recursively construct explicit elements of provably high order
in finite fields.  We do this using the recursive formulas developed by Elkies 
to describe explicit modular towers. 
In particular, we give two 
explicit constructions based on two examples of 
his formulas and demonstrate that the 
resulting elements have high order.  
Between the two constructions, we are able to generate high order 
elements in every characteristic.
Despite the use of the modular recursions of Elkies, our methods are quite 
elementary and require no knowledge of modular curves.
We compare our results to a recent result of Voloch.  
In order to do this, we state and prove a slightly more refined version of a 
special case of his result.
\end{abstract}

\maketitle

\section{Introduction}
Finding large order elements of finite fields has long been a problem of 
interest, particularly to  cryptographers.  Given a finite field $\gf_q$, 
Gao~\cite{gao:1999} gives an algorithm for constructing elements of $\gf_{q^n}$
of order greater than 
$$n^{\frac{\log_qn}{4\log_q(2\log_qn)}-\frac{1}{2}}.$$
The advantage of the algorithm is that it makes no restriction on $q$ and it 
allows one to produce a provably
high order element in any desired extension of $\gf_q$ provided that
one can find a polynomial in $\gf_q[x]$ with certain desirable properties.
Gao conjectures that for any $n>1$, there exists a polynomial of degree 
at most $2\log_qn$ satisfying the conditions of his theorem.  
Conflitti has made some improvement to Gao's construction in~\cite{Con:2001}.
However, the aforementioned conjecture remains unproven.  Another 
result concerning the $q$ ``shifts" of an element of a general 
extension of $\gf_q$ appears in~\cite[Corollary 4.4]{vzGS:2001}.

For special finite fields, it is possible to construct elements which 
can be proved to have much higher orders.  
For example, in Theorems~\ref{2tower} and~\ref{3tower} of this paper we 
construct elements of higher order in extensions 
of $\gf_q$ of the form $\gf_{q^{2^n}}$ and $\gf_{q^{3^n}}$.
See~\cite{GV:1995, GvGP:1998, vGS:1995} on orders of Gauss 
periods and~\cite{cheng:2005, cheng:2007} on Kummer extensions.
It has been pointed out to us that the method of~\cite{cheng:2005, cheng:2007} 
is able to produce higher order elements in the same extensions as our method.
However, our method of construction is new, and we hope that it will 
prove to be a fruitful technique.

In~\cite{voloch:2007}, Voloch shows that under certain conditions, one of the 
coordinates of a point on a plane curve must have high order.
The bounds we obtain through our methods have order of
magnitude similar to those predicted in the main theorem of~\cite{voloch:2007}.
In a special case however, Voloch is able to achieve bounds which are much 
better. 
See section 5 of~\cite{voloch:2007}.  
Unfortunately, Voloch does not fully state this theorem and only alludes to 
how one may adapt the proof of his main theorem for this special case.
The bounds given in~\cite{voloch:2007} are not as explicit as the ones given
in this paper.  Moreover, Voloch gives no explicit examples of his theorems.  
In Section~\ref{voloch_comparison} of this paper, 
we apply Voloch's technique to obtain a 
more explicit version of the special case of his main theorem.
We then construct a sequence of elements for which his bounds apply and compare
with our methods.

In this paper, we consider elements in finite field towers recursively generated
according to the equations for explicit modular towers~\cite{elk:1998}. 
We give two explicit 
constructions: one for odd characteristic and one for characteristic not equal 
to $3$.  In the first case, we explicitly construct elements 
of $\gf_{q^{2^n}}$
whose orders are bounded below by
$2^{\frac{1}{2}n^2+\frac{3}{2}n+\ord_2(q-1)-1}$.
In the second, we obtain elements of $\gf_{q^{3^n}}$ whose orders are 
bounded below by
$3^{\frac{1}{2}n^2+\frac{3}{2}n+\ord_3(q-1)}$.
Throughout we use the convention that exponentiation is 
right-associative, i.e., $a^{b^c}:=a^{(b^c)}$.

\section{Constructions Arising from Modular Towers}

In~\cite{elk:1998}, Elkies gives a recursive formula for the defining equations 
of the modular curve $X_0(\ell^n)$ by identifying $X_0(\ell^n)$ 
within the product $\left(X_0(\ell^2)\right)^{n-1}$ 
for $n>1$.  For several cases, he even writes explicit equations.
For example, in the case $\ell=2$, the recursion is governed by the rule
\begin{equation}\label{x02}
(x^2_j-1)\left(\left(\frac{x_{j+1}+3}{x_{j+1}-1}\right)^2-1\right)=1
\mbox{ for } j=1,2,\dots, n-2.
\end{equation}
Elkies also notices that under a suitable change of variables and a reduction 
modulo 3, the equation becomes 
$$y_{j+1}^2=y_j-y_j^2,$$
which was used by Garcia and Stichtenoth~\cite{GS:1996} to recursively 
construct 
an asymptotically optimal function field tower.  In fact, Elkies notes that 
many recursively constructed optimal towers may now be seen as arising from 
these modular curve constructions and speculates that perhaps all such 
towers are modular in this sense.

In this paper, we use Elkies' formulas to generate high order elements in 
towers of finite fields.  For example, 
the following construction will yield high order elements in odd characteristic.
The equation (\ref{x02}) may be 
manipulated to the form $f(X,Y)=0$, where
\begin{equation}\label{2towerpoly}
f(X,Y):=Y^2+(6-8X^2)Y+(9-8X^2), 
\end{equation}
and we have made the substitution $X=x_j$ and $Y=x_{j+1}$.  
Now, choose 
$q=p^m$ to be an odd prime power such that $\gf_q$ contains the fourth roots of 
unity (i.e. $q\equiv 1\pmod 4$). Choose $\alpha_0\in\gf_q$ such that 
$\alpha_0^2-1$ is not a square in $\gf_q$.  In Lemma~\ref{2start}
(see Section~\ref{numbthysec}), 
we will show that such an $\alpha_0$ always exists.  Finally, 
define $\alpha_n$ by $f(\alpha_{n-1},\alpha_n)=0$ for 
$n\ge 1$. This construction yields the following result; where, as usual, for 
a prime $\ell$, $\ord_\ell(a)$ denotes the highest power of $\ell$ dividing $a$.
\begin{thm}\label{2tower}
Let $\delta_n:=\alpha_n^2-1$.
Then $\delta_n$ has degree $2^n$ over $\gf_q$, and 
the order of $\delta_n$ in $\gf_{q^{2^n}}$ is greater than
$2^{\frac{1}{2}n^2+\frac{3}{2}n+\ord_2(q-1)}$ unless $q\equiv 2\pmod 3$ and 
$\alpha_0=\pm\left(\frac{p-1}{2}\right)$, 
in which case the order of $\delta_n$ is greater than
$2^{\frac{1}{2}n^2+\frac{3}{2}n+\ord_2(q-1)-1}$.
\end{thm}

To accommodate even characteristic, we have also considered Elkies' formula for 
$X_0(3^n)$.  We will prefer to work with the equation in the 
polynomial form $g(X,Y)=0$, where
\begin{equation}\label{3towerpoly}
g(X,Y):=Y^3+(6-9X^3)Y^2+(12-9X^3)Y+(8-9X^3).
\end{equation}
For this construction, choose $q$ to be a prime power congruent to $1$ modulo 
$3$ but not equal to $4$.
The condition $q\equiv 1\pmod 3$ assures the presence of the 
third roots of unity in $\gf_q$.
Choose $\beta_0\in\gf_q$ such that $\beta_0^3-1$ is not a 
cube in $\gf_q$.  In Lemma~\ref{3start}
(see Section~\ref{numbthysec}), we show that such a $\beta_0$ always 
exists except when $q=4$.
Finally, define $\beta_n$ by $g(\beta_{n-1},\beta_n)=0$ for $n\ge 1$.
For this construction, we have the following result.
\begin{thm}\label{3tower}
Let $\gamma_n:=\beta_n^3-1$.
Then $\gamma_n$ has degree $3^n$ over $\gf_q$, 
and the order of $\gamma_n$ in $\gf_{q^{3^n}}$ is greater than
$3^{\frac{1}{2}n^2+\frac{3}{2}n+\ord_3(q-1)}$.
\end{thm}
There are two interesting things about the above constructions.  
The first is
that, computationally, the elements $\delta_n$ and $\gamma_n$ appear to have 
much higher order than our bounds suggest.  See Section~\ref{eg} for examples.
The second interesting thing
is that, as with the case of 
the optimal function field tower constructions of Garcia and Stichtenoth 
\cite{GS:1995,GS:1996}
arising from these modular curve recipes, our proofs 
do not at all exploit this modularity.  
Perhaps the key to achieving better bounds lies in this relationship.

The paper is organized as follows.  In Section~\ref{numbthysec}, 
we will state and prove some 
elementary number theory facts that will be of use to us.  
In Section~\ref{quadsect}, we 
consider the first construction; and in Section~\ref{cubicsect}, 
we consider the second.  Finally, in Section~\ref{eg}, we give a few examples 
of each of the main theorems.

\section{Number Theoretic Facts}\label{numbthysec}

Recall the following well known fact for detecting perfect $n$-th powers in 
finite fields.  See~\cite[p. 81]{IR:1990} for example.
\begin{fact}\label{pwrdetect}
If $q\equiv 1\pmod n$, then
$x\in\gf_q^*$ is a perfect $n$-th power if and only if 
$x^{(q-1)/n}=1$.
\end{fact}
Also recall the following facts, which can be easily proved.
\begin{fact}\label{gpfact}
Let $x\in\gf_q^*$ of multiplicative order $d$. 
For $m,n\in\nn$, if $x^n\neq 1$ and $x^{nm}=1$, then $\gcd(d,m)>1$.
\end{fact}
\begin{fact}\label{ellpwr}
Let $x\in\gf_q^*$ of multiplicative order $d$.  If 
$\ell$ is a prime, $m=\ord_\ell(n)$, and $x^n$ is a nontrivial 
$\ell$-th root of unity, then $\ell^{m+1}$ divides $d$.
\end{fact}

The following lemmas are useful for bounding the orders of the elements 
appearing in Theorems~\ref{2tower} and~\ref{3tower}.
\begin{lma}\label{gcdcomp}
Let $\ell,b\in\nn$ such that $b\equiv 1\pmod\ell$, and let $M,N\in\nn$  
with $M<N$.
Then
$$\gcd\left(\sum_{j=1}^\ell b^{\ell^M(\ell-j)},
\sum_{j=1}^\ell b^{\ell^N(\ell-j)}\right)=\ell;$$
and hence 
$\displaystyle\frac{1}{\ell}\sum_{j=1}^\ell b^{\ell^M(\ell-j)}$ and 
$\displaystyle\frac{1}{\ell}\sum_{j=1}^\ell b^{\ell^N(\ell-j)}$ 
are coprime.
\end{lma}
\begin{proof}
The following computation follows from Euclid's algorithm:
\begin{equation}\label{gcd1}
\gcd\left(\sum_{j=1}^\ell b^{\ell^N(\ell-j)},b^{\ell^N}-1\right)
=\gcd\left(\ell,b^{\ell^N}-1\right)=\ell.
\end{equation}
Since $M<N$, repeatedly using the difference of $\ell$-th powers formula shows 
that
$\sum_{j=1}^\ell b^{\ell^M(\ell-j)}$ divides $b^{\ell^N}-1$.  
Also, since $b\equiv 1\pmod \ell$, it is clear that $\ell$ divides both 
$\sum_{j=1}^\ell b^{\ell^M(\ell-j)}$ and $\sum_{j=1}^\ell b^{\ell^N(\ell-j)}$. 
Therefore,
$$\gcd\left(\sum_{j=1}^\ell b^{\ell^M(\ell-j)},
\sum_{j=1}^\ell b^{\ell^N(\ell-j)}\right)=\ell.$$
\end{proof}
\begin{lma}\label{primebound}
Let $\ell,b,N\in\nn$ with $\ell$ prime  and $b\equiv 1\pmod\ell$.  
If $p$ is a prime dividing 
$\displaystyle\frac{1}{\ell}\sum_{j=1}^\ell b^{\ell^N(\ell-j)}$, then 
$p>\ell^{N+1}$.
\end{lma}
\begin{proof}
Since $\ell\ge 2$ and $b\equiv 1\pmod\ell$, $\ell^2$ divides $(b^{\ell^N}-1)$. 
Hence, $p\neq\ell$ for otherwise, we have a contradiction with 
(\ref{gcd1}).  Thus, $p$ dividing 
$\frac{1}{\ell}\sum_{j=1}^\ell b^{\ell^N(\ell-j)}$ implies that 
$\sum_{j=1}^\ell b^{\ell^N(\ell-j)}\equiv 0\pmod p$.  So, $b^{\ell^N}$ is a 
nontrivial 
$\ell$-th root of unity modulo $p$.  Therefore, by 
Fact~\ref{ellpwr}, $\ell^{N+1}$ 
divides $p-1$, and hence $p>\ell^{N+1}$.
\end{proof}

The following two lemmas essentially give the necessary and sufficient 
conditions for completing the 
first step in the construction of our towers, i.e., under certain restrictions 
on $q$, they demonstrate the 
existence of $\alpha_0$ and $\beta_0$ each having its desired property.
The proofs involve counting $\gf_q$ solutions to equations via character sums.
We refer the reader to~\cite[Chapter 8]{IR:1990} for more on this technique.
As in~\cite{IR:1990}, for characters $\psi$ and $\lambda$ on $\gf_q$, we 
denote the Jacobi sum of $\psi$ and $\lambda$ by 
$J(\psi,\lambda):=\sum_{a+b=1}\psi(a)\lambda(b)$.
\begin{lma}\label{2start}
Let $q$ be a prime power.  Then there exists $\alpha_0\in\gf_q$ such that 
$\delta_0=\alpha_0^2-1$ is not a square in $\gf_q$ if and only if 
$q$ is odd.
\end{lma}
\begin{proof}
First, note that if $q$ is even, then every element of $\gf_q$ is a square. 
So, we assume that $q$ is odd.
We desire $\alpha_0\in\mathbb{F}_q^*$ such that $\alpha_0^2-1$ is not a square. 
Our method for proving that such an $\alpha_0$ exists involves counting 
solutions
to the equation $x^2-y^2=1$.  
Let $\tau$ be the unique character of exact order 2 on $\mathbb{F}_q$. Then
\begin{eqnarray*}
\#\{(x,y)\in\mathbb{F}_q^2:x^2-y^2=1\}
&=&\sum_{\substack{a,b\in\gf_q,\\ a+b=1}}
\left(\sum_{j=0}^1{\tau^j(a)}\right)\left(\sum_{j=0}^1{\tau^j(-b)}\right)\\
&=&\sum_{i=0}^{1} \sum_{j=0}^{1} \tau^j(-1)J(\tau^i,\tau^j)\\
&=&q+\tau(-1)J(\tau,\tau)
=q-1.
\end{eqnarray*}

On the other hand, if $\alpha_0^2-1$ is a square for all choices of $\alpha_0$,
then $\alpha_0^2-1=y^2$ has a solution for all
$\alpha_0\in\mathbb{F}_q$. In this case, we have
\begin{eqnarray*}
\#\{(x,y)\in\mathbb{F}_q^2:x^2-y^2=1\}
&=&\sum_{\alpha_0\in\mathbb{F}_q} \#\{y\in\gf_q:y^2=\alpha_0^2-1\}\\
&=&\sum_{\alpha_0^2=1}1+\sum_{\alpha_0^2\neq 1}2=2+2(q-2)=2q-2.
\end{eqnarray*}
Thus, the assumption that $\alpha_0^2-1$ is always a square leads to the 
conclusion $q-1=2q-2$, which implies $q=1$, a contradiction.
\end{proof}

\begin{lma}\label{3start}
Let $q$ be a prime power.  Then there exists 
$\beta_0\in\gf_q$ such that $\gamma_0=\beta_0^3-1$ is not a cube in $\gf_q$
if and only if $q\equiv 1\pmod 3$ and $q\neq 4$.
\end{lma}
\begin{proof}
First, note that if $q\not\equiv 1\pmod 3$, then every element of $\gf_q$ is 
a cube.  So, we will assume that $q\equiv 1\pmod 3$.  As mentioned earlier, 
this means that $\gf_q$ contains a primitive third root of unity.
We now count $\gf_q$ 
solutions to the equation $x^3-y^3=1$.  Let $\chi$ be any character of 
order $3$ on $\gf_q$.
\begin{eqnarray*}
\#\{(x,y)\in\gf_q^2:x^3-y^3=1\}
&=&\sum_{\substack{a,b\in\gf_q,\\ a+b=1}}\left(\sum_{j=0}^2\chi^j(a)\right)
\left(\sum_{j=0}^2\chi^j(-b)\right)\\
&=&\sum_{i=0}^2\sum_{j=0}^2\chi^{j}(-1)J(\chi^i,\chi^j)\\
&=&q-2\chi(-1)+J(\chi,\chi)+J(\chi^2,\chi^2)\\
&=&q-2+2\mathrm{Re}J(\chi,\chi).
\end{eqnarray*}
On the other hand, if we assume that $\beta_0^3-1$ is a cube for all choices of 
$\beta_0\in\gf_q$, then
\begin{eqnarray*}
\#\{(x,y)\in\gf_q^2:x^3-y^3=1\}
&=&\sum_{\beta_0\in\gf_q}\#\{y\in\gf_q:\beta_0^3-y^3=1\}\\
&=&\sum_{\beta_0^3=1}1+\sum_{\beta_0^3\neq 1}3
=3+3(q-3)=3q-6.
\end{eqnarray*}
Thus, the assumption that $\beta_0^3-1$ is always a cube leads to the 
conclusion that $|2q-4|=|(3q-6)-(q-2)|=|2\mathrm{Re}J(\chi,\chi)|\le 2\sqrt q$, 
which implies $|q-2|\le\sqrt q$.  
This implies that $(q-1)(q-4)\le 0$.
The only $q\equiv 1\pmod 3$ satisfying this inequality is $q=4$.
\end{proof}

\section{The Quadratic Tower for Odd Characteristic}\label{quadsect}
In this section, we consider the first tower, which is 
recursively constructed using~\eqref{2towerpoly}.  Throughout 
this section we will assume that $p$ is an odd prime and that 
$q=p^m\equiv 1\pmod 4$.  In particular, 
if $p\equiv 3\pmod 4$, then
$2|m$.  As discussed in the introduction, this condition ensures the existence 
of a primitive fourth root of unity.  This will be seen to be a necessary 
ingredient in the construction of our tower.  
We also fix $\alpha_0$ such that $\delta_0=\alpha_0^2-1$ is not a 
square in $\gf_q$.  Recall that that Lemma~\ref{2start} ensures the existence 
of such an $\alpha_0$.

Before moving forward, we need to establish the relationship between $\delta_n$
and $\delta_{n-1}$.  
From~\eqref{2towerpoly} and the definition of $\delta_n$ 
(see Theorem~\ref{2tower}), we deduce 
that $\delta_{n-1}$ and $\delta_n$ are related by 
$F(\delta_{n-1},\delta_n)=0$ ($n\ge 1$), where 
\begin{equation}\label{deltarel}
F(X,Y):=Y^2-(48X+64X^2)Y-64X.
\end{equation}
We also fix the following more compact notation for the norm. We take 
\begin{eqnarray*}
\N_{n,j}:\gf_{q^{2^n}}&\rightarrow&\gf_{q^{2^{n-j}}},\\
\alpha&\mapsto&\alpha^{\prod_{k=1}^{j}(q^{2^{n-k}}+1)}.
\end{eqnarray*}

For the purpose of making the proof easier to 
digest, we break Theorem~\ref{2tower} into a pair of propositions.
\begin{prop}\label{2deg}
The elements $\alpha_n$ and $\delta_n$ have degree $2$ over $\gf_{q^{2^{n-1}}}$ 
for $n\ge 1$.
\end{prop}
\begin{proof}
First note that the discriminant of $f(\alpha_{n-1},Y)$ is $\delta_{n-1}
=\alpha_{n-1}^2-1$ for all $n\ge 1$.  
We will proceed by induction on $n$.  
Recall that $\alpha_0$ was chosen so that $\delta_0$,
the discriminant of $f(\alpha_0,Y)$, is not a square in $\gf_q$.  
Thus, $\alpha_1$ satisfies 
an irreducible polynomial of degree $2$ over $\gf_q$, i.e., $\alpha_1$ has 
degree $2$ over $\gf_q$.
We may take $\{1,\alpha_1\}$ as a basis for $\gf_q(\alpha_1)$ over $\gf_q$.  
Writing $\delta_1$ in terms of the basis, we 
have $\delta_1=\alpha_1^2-1=(8\alpha_0^2-6)\alpha_1+(8\alpha_0^2-10)$.
So, $\delta_1\in\gf_q$ if and only if $8\alpha_0^2-6=0$.
If $8\alpha_0^2-6=0$, then $\delta_0=\alpha_0^2-1=-4^{-1}$, which is a 
square in $\gf_q$ since $\gf_q$ contains the fourth roots of unity.
This is contrary to our choice of $\alpha_0$.
Thus, $\delta_1$ has degree $2$ over $\gf_q$ as well.

Now, suppose that $\alpha_{k}$ and $\delta_k$ both
have degree $2$ over $\gf_{q^{2^{k-1}}}$ for $1\le k\le n$.  
Then $f(\alpha_{n-1},Y)$ is the minimum polynomial of $\alpha_n$ over 
$\gf_{q^{2^{n-1}}}$; and hence, the discriminant 
is not a square in $\gf_{q^{2^{n-1}}}$.  In particular, 
\begin{equation}\label{discnotsq}
\delta_{n-1}^{(q^{2^{n-1}}-1)/2}=-1.
\end{equation}
Observe that $F(\delta_{n-1},Y)$ is the minimum polynomial of 
$\delta_n$ over $\gf_{q^{2^{n-1}}}$.  To prove that the degree of 
$\alpha_{n+1}$ 
over $\gf_{q^{2^n}}$ is 2, we show that $f(\alpha_n,Y)$ is irreducible 
over $\gf_{q^{2^n}}$.  
Now,
\begin{eqnarray*}
\delta_n^{(q^{2^n}-1)/2}
&=&\left(\delta_n^{(q^{2^{n-1}}+1)}\right)^{(q^{2^{n-1}}-1)/2}
=\left(\mathrm{N}_{n,1}(\delta_n)\right)^{(q^{2^{n-1}}-1)/2}\\
&=&\left(-64\delta_{n-1}\right)^{(q^{2^{n-1}}-1)/2}
=-1.
\end{eqnarray*}
Here we have used~\eqref{discnotsq} and the fact that $-64$ is a square in 
$\gf_{q^{2^{n-1}}}$ since $\gf_q$ contains the fourth roots of unity.
Thus, $\delta_n$ is not a square, and hence $f(\alpha_n,Y)$ is irreducible.
So, the set $\{1,\alpha_{n+1}\}$ forms a basis for $\gf_{q^{2^{n+1}}}$ over 
$\gf_{q^{2^{n}}}$.
Now, we write $\delta_{n+1}$ in terms of the basis, 
and apply the same argument 
as for $\delta_1$ to demonstrate that the degree 
of $\delta_{n+1}$ over $\gf_{q^{2^{n}}}$ is $2$ as well.
This completes the induction and the proof.
\end{proof}

An easy induction proof, exploiting the fact that $F(\delta_{k-1},Y)$ is 
the minimum polynomial of $\delta_k$ over $\gf_{q^{2^{k-1}}}$ for $1\le k\le n$,
shows that 
\begin{equation}\label{normcomp}
\N_{n,j}(\delta_n)=(-64)^{(2^{j}-1)}\delta_{n-j}
\end{equation}
for $1\le j\le n$.  This fact will be useful in the proof of the proposition 
below.
\begin{prop}\label{2order}
The order of $\delta_n$ in $\gf_{q^{2^n}}$ is greater than 
$2^{\frac{1}{2}n^2+\frac{3}{2}n+\ord_2(q-1)}$ unless $q\equiv 2\pmod 3$ and 
$\alpha_0=\pm\left(\frac{p-1}{2}\right)$,
in which case the order of $\delta_n$
is greater than
$2^{\frac{1}{2}n^2+\frac{3}{2}n+\ord_2(q-1)-1}$.
\end{prop}
\begin{proof}
We first compute the power of $2$ dividing the order of $\delta_n$.
Recall from the proof of Proposition~\ref{2deg}
that $\delta_n^{(q^{2^n}-1)/2}\neq 1$; but of course, 
$\delta_n^{(q^{2^n}-1)}=1$ since $\delta_n\in\gf_{q^{2^n}}$.  
Since $q\equiv 1\pmod 4$, 
$\ord_2(q^{2^{j}}+1)=1$
for each $j\ge 1$.  
Repeatedly using the difference of 
squares formula, we have
\begin{eqnarray*}
\ord_2\left(\frac{q^{2^n}-1}{2}\right)
&=&\ord_2(q-1)-1+\sum_{j=0}^{n-1}\ord_2(q^{2^j}+1)\\
&=&n-1+\ord_2(q-1).
\end{eqnarray*}
Thus, $2^{n+\ord_2(q-1)}$ divides the order of $\delta_n$ by Fact~\ref{ellpwr}.

Now we look for odd primes dividing the order. 
By Fact~\ref{gpfact}, the order of $\delta_n$ has a common factor with 
$(q^{2^{n-j}}+1)/2$ for each $j$ such that the
$\frac{(q^{2^n}-1)}{(q^{2^{n-j}}+1)/2}$ power of $\delta_n$ is not equal to $1$.
%\begin{equation*}
%\delta_n^{\frac{(q^{2^n}-1)}{(q^{2^{n-j}}+1)/2}}\neq 1.
%\end{equation*}
By~\eqref{normcomp}, we have that the 
$\frac{(q^{2^n}-1)}{(q^{2^{n-j}}+1)/2}$ power of $\delta_n$ is equal to
\begin{eqnarray*}
%\delta_n^{\frac{(q^{2^n}-1)}{(q^{2^{n-j}}+1)/2}}&=&
\left(\mathrm{N}_{n,j-1}(\delta_{n})\right)^{2(q^{2^{n-j}}-1)}
=((-64)^{(2^{(j-1)}-1)}\delta_{n-j+1})^{2(q^{2^{n-j}}-1)}
=(\delta_{n-j+1})^{2(q^{2^{n-j}}-1)}
\neq 1
\end{eqnarray*}
provided that 
$\delta_{n-j+1}^2\not\in\gf_{q^{2^{n-j}}}$.
From~\eqref{deltarel}, we know that we may write $\delta_{n-j+1}^2$ as 
\begin{equation*}
\delta_{n-j+1}^2=(48\delta_{n-j}+64\delta_{n-j}^2)\delta_{n-j+1}+64\delta_{n-j}.
\end{equation*}
Thus, $\delta_{n-j+1}^2\in\gf_{q^{n-j}}$ if and only if $\delta_{n-j}$ 
satisfies the equation $48\delta_{n-j}+64\delta_{n-j}^2=0$.
If this were the case, then $\delta_{n-j}=0$ or $\delta_{n-j}=-3^{-1}4$.
By Proposition~\ref{2deg}, this implies that $n=j$.  However, $\delta_0=0$ 
contradicts the choice of $\alpha_0$; and $\delta_0=-4^{-1}3$ contradicts 
the choice of $\alpha_0$ unless $-3$ is not a perfect square, that is, unless 
$q\equiv 2\pmod 3$.  If $q\equiv 2\pmod 3$, then the only choices of $\alpha_0$ 
that give $\delta_0=-4^{-1}3$ are $\alpha_0=\pm\left(\frac{p-1}{2}\right)$.
Thus, the order of $\delta_n$ has a common factor with $(q^{2^{n-j}}+1)/2$ 
for each $1\le j\le n$ unless $q\equiv 2\pmod 3$, 
$\alpha_0=\pm\left(\frac{p-1}{2}\right)$, and $j=n$.
Each of these factors must be odd since 
$\ord_2(q^{2^{n-j}}+1)=1$ as noted above.
By Lemma~\ref{gcdcomp} with $\ell=2$ and $b=q$, we see that these factors 
must be pairwise coprime as well.
Hence, we get either $n$ or $n-1$ 
distinct odd prime factors dividing the order 
of $\delta_n$ depending on the case.  
By Lemma~\ref{primebound}, each such prime factor 
must bounded below by $2^{n-j+1}$.  Therefore,
the order of $\delta_n$ is bounded below by
\begin{eqnarray*}
2^{n+\ord_2(q-1)}\prod_{j=1}^{n}2^{n-j+1}
&=&2^{n+\ord_2(q-1)+n(n+1)/2}
=2^{\frac{n^2+3n}{2}+\ord_2(q-1)}
\end{eqnarray*}
unless $q\equiv 2\pmod 3$ and $\alpha_0=\pm\left(\frac{p-1}{2}\right)$, in 
which case the order is bounded below by 
$2^{\frac{1}{2}n^2+\frac{3}{2}n+\ord_2(q-1)-1}$.
\end{proof}
Theorem~\ref{2tower} follows by combining the two propositions.
The authors would like to point out that it is possible to achieve a slightly 
better lower bound for the order of
$\delta_n$ by the following method.  First, choose a 
square root of $\delta_{n-1}$, say $\sqrt{\delta_{n-1}}\in\gf_{q^{2^n}}$.  Then 
use the method above to prove a lower bound for the order of 
$\sqrt{\delta_{n-1}}$.  Finally,
deduce a bound for the order of $\delta_n$.  The improvement, however, only 
affects the coefficient of $n$ in the exponent.  Since computationally our 
bounds do not appear to be that close to the truth, we have decided to work 
directly with $\delta_n$ instead.

\section{The Cubic Tower for Characteristic not 3}\label{cubicsect}

In this section, we consider the second tower, which is recursively constructed 
using~\eqref{3towerpoly}.
Recall that, for this tower, we assume that $q\equiv 1\pmod 3$ and $q\neq 4$.
This means that $\gf_q$ will contain the third roots of unity, and hence the 
third roots of $-1$ as well.  We also fix a $\beta_0$ such that 
$\gamma_0=\beta_0^3-1$ is not a cube in $\gf_q$.  Recall that Lemma~\ref{3start}
ensures the existence of such a $\beta_0$.

Before we begin the proof of Theorem~\ref{3tower}, we need to 
establish the relationship between $\gamma_{n-1}$ and 
$\gamma_n$.  The relationship is given by 
$G(\gamma_{n-1},\gamma_n)=0$ for $n\ge 1$, where
\begin{equation}\label{gammarel}
G(X,Y):=Y^3-(270X+972X^2+729X^3)Y^2-(972X+729X^2)Y-729X.
\end{equation} 
This follows from~\eqref{3towerpoly} and the definition of $\gamma_n$.
We also fix the following notation for the norm.
\begin{eqnarray*}
\mathrm{N}_{n,j}:\gf_{q^{3^n}}&\rightarrow&\gf_{q^{3^{n-j}}},\\
\beta&\mapsto&
\beta^{\prod_{k=1}^j\left(\left(q^{3^{n-k}}\right)^2+q^{3^{n-k}}+1\right)}.
\end{eqnarray*}

As in section~\ref{quadsect}, we break the result into two smaller propositions.
\begin{prop}\label{3deg}
The elements $\beta_n$ and $\gamma_n$ both have degree $3$ over 
$\gf_{q^{3^{n-1}}}$ for $n\ge 1$.
\end{prop}
\begin{proof}
By carefully examining the cubic formula applied to the polynomial, one 
observes that $g(\beta_{n-1},Y)$ is irreducible if and only if 
$\gamma_{n-1}=\beta_{n-1}^3-1$ is not a cube in $\gf_{q^{3^{n-1}}}$.  
Thus, $\beta_n$ will have degree $3$ over $\gf_{q^{3^{n-1}}}$
if and only if $\gamma_{n-1}$ is not a cube in $\gf_{q^{3^{n-1}}}$ 
for all $n\ge 1$.
As with the proof of Proposition~\ref{2deg}, we proceed by induction on $n$.  
Recall that $\beta_0$ was chosen so that $\gamma_0$ is not a cube in $\gf_q$.
Thus, $\beta_1$ has degree $3$ 
over $\gf_q$.  So, we may take $\{1,\beta_1,\beta_1^2\}$ as a basis for 
$\gf_{q^3}$ over $\gf_q$.  Writing $\gamma_1$ in terms of the basis, we have
\begin{equation*}
\gamma_1=\beta_1^3-1
=(9\beta_{0}^3-6)\beta_1^2+(9\beta_{0}^3-12)\beta_1+(9\beta_{0}^3-9).
\end{equation*}
So, $\gamma_1\in\gf_q$ if and only if $9\beta_0^3-6=0$ and $9\beta_0^3-12=0$.
This leads to the conclusion that $\gamma_0=-3^{-1}$ and $\gamma_0=3^{-1}$, 
which implies that $2=0$, i.e., the 
characteristic is $2$.  In this case, we are led to the conclusion that
$\gamma_0=1$, which is a cube.  This of course is contrary to our choice of 
$\gamma_0$.  Therefore, $\gamma_1\not\in\gf_q$, i.e., the degree of $\gamma_1$ 
over $\gf_q$ is $3$.  This completes the trivial case.

Now, let $\omega$ be a primitive cube root of unity in $\gf_q$ and 
suppose that $\beta_k$ and $\gamma_k$ both have degree $3$ over 
$\gf_{q^{3^{k-1}}}$ for $1\le k\le n$.  Then $g(\beta_{n-1},Y)$ is the
minimum polynomial of $\beta_n$ over $\gf_{q^{3^{n-1}}}$; and hence 
$\gamma_{n-1}$ is not a cube in $\gf_{q^{3^{n-1}}}$.  In particular,
\begin{equation*}
\gamma_{n-1}^{(q^{3^{n-1}}-1)/3}=\omega.
\end{equation*}
Observe that
$G(\gamma_{n-1},Y)$ is the minimum polynomial of $\gamma_n$ over 
$\gf_{q^{3^{n-1}}}$.  Thus,
\begin{eqnarray*}
\gamma_n^{(q^{3^n}-1)/3}
&=&\left(
\gamma_n^{\left(\left(
q^{3^{n-1}}\right)^2+q^{3^{n-1}}+1\right)}\right)^{(q^{3^{n-1}-1})/3}
=\left(\mathrm{N}_{n,1}(\gamma_n)\right)^{(q^{3^{n-1}}-1)/3}\\
&=&\left(-729\gamma_{n-1}\right)^{(q^{3^{n-1}}-1)/3}=\omega;
\end{eqnarray*}
i.e., $\beta_{n+1}$ has degree $3$ over $\gf_{q^{3^n}}$.
To prove that $\gamma_{n+1}$ also has degree $3$ over $\gf_{q^{3^n}}$, write 
$\gamma_{n+1}$ in terms of the $\gf_{q^{3^n}}$-basis 
$\{1,\beta_{n+1},\beta_{n+1}^2\}$, and proceed as we did for $\gamma_1$.
\end{proof}

An easy induction proof using the fact that $G(\gamma_{k-1},Y)$ is the 
minimum polynomial of $\gamma_k$ over $\gf_{q^{3^{k-1}}}$ for 
$1\le k\le n$, shows that
\begin{equation*}
\mathrm{N}_{n,j}(\gamma_n)=(-729)^{(3^j-1)}\gamma_{n-j}
\end{equation*}
for $1\le j\le n$.
\begin{prop}
The order of $\gamma_n$ in $\gf_{q^{3^n}}$ is greater than 
$3^{\frac{1}{2}n^2+\frac{3}{2}n+\ord_3(q-1)}$.
\end{prop}
\begin{proof}
We first compute the power of $3$ dividing the order of $\gamma_n$.  Recall 
from the proof of Proposition~\ref{3deg} that $\gamma_n^{(q^{3^n}-1)/3}\neq 1$.
However, $\gamma_n^{(q^{3^n}-1)}=1$ since $\gamma_n\in\gf_{q^{3^n}}$.
Since $q\equiv 1\pmod 3$, $\ord_3((q^{3^j})^2+q^{3^j}+1)
=1$ for each $j\ge 1$.
Repeatedly using the difference of cubes formula,
we have
\begin{eqnarray*}
\ord_3\left(\frac{q^{3^n}-1}{3}\right)
&=&\ord_3(q-1)-1
+\sum_{j=0}^{n-1}\ord_3\left(\left(q^{3^j}\right)^2+q^{3^j}+1\right)\\
&=&n-1+\ord_3(q-1).
\end{eqnarray*}
Thus, $3^{n+\ord_3(q-1)}$ divides the order of $\gamma$ by Fact~\ref{ellpwr}.

Now, we look for primes dividing the order that are not equal to $3$.  In 
particular, we will show that the order of $\gamma_n$ has a common factor
with $((q^{3^{n-j}})^2+q^{3^{n-j}}+1)/3$ for each $1\le j\le n$.
This factor must not be a multiple of $3$ since 
$\ord_3((q^{3^{n-j}})^2+q^{3^{n-j}}+1)=1$ as noted above.
By Lemma~\ref{gcdcomp}, with $\ell=3$ and $b=q$, we see that these factors 
must be pairwise coprime as well.  Hence, we get $n$ distinct prime factors 
dividing the order of $\gamma_n$, none of which are equal to $3$.
By Lemma~\ref{primebound}, each of these primes must be bounded below by
$3^{n-j+1}$.  Hence, if we can show that the order of $\gamma_n$ has a 
common factor
with $((q^{3^{n-j}})^2+q^{3^{n-j}}+1)/3$ for $1\le j\le n$, 
then we have that the order of 
$\gamma_n$ is bounded below by
\begin{eqnarray*}
3^{n+\ord_3(q-1)}\prod_{j=1}^n3^{n-j+1}
&=&3^{n+\ord_3(q-1)+n(n+1)/2}
=3^{\frac{n^2+3n}{2}+\ord_3(q-1)}.
\end{eqnarray*} 
By Fact~\ref{gpfact}, the proof will be complete when we show that
the $\frac{q^{3^n}-1}{((q^{3^{n-j}})^2+q^{3^{n-j}}+1)/3}$ power of 
$\delta_n$ is not equal to $1$ for $1\le j\le n$.
%\begin{equation*}
%\gamma_n^{\frac{q^{3^n}-1}{((q^{3^{n-j}})^2+q^{3^{n-j}}+1)/3}}\neq 1
%\mbox{ for } 1\le j\le n.
%\end{equation*}
Now, $\delta_n$ raised to the 
$\frac{q^{3^n}-1}{((q^{3^{n-j}})^2+q^{3^{n-j}}+1)/3}$ power is equal to
\begin{eqnarray*}
%\gamma_n^{\frac{q^{3^n}-1}{((q^{3^{n-j}})^2+q^{3^{n-j}}+1)/3}}&=&
(\mathrm{N}_{n,j-1}(\gamma_n))^{3(q^{3^{n-j}-1})}
=((-729)^{(3^{(j-1)}-1)}\gamma_{n-j+1})^{3(q^{3^{n-j}}-1)}
\neq 1
\end{eqnarray*}
provided $\gamma_{n-j+1}^3\not\in\gf_{q^{3^{n-j}}}$.
From (\ref{gammarel}), we know that we may write $\gamma_{n-j+1}^3$ as
\begin{equation*}
\gamma_{n-j+1}^3
=(270\gamma_{n-j}+972\gamma_{n-j}^2+729\gamma_{n-j}^3)\gamma_{n-j+1}^2
+(972\gamma_{n-j}+729\gamma_{n-j}^2)\gamma_{n-j+1}
+729\gamma_{n-j}.
\end{equation*}
Thus, $\gamma_{n-j+1}^3\in\gf_{q^{3^{n-j}}}$ if only if $\gamma_{n-j}$ 
satisfies the system
\begin{eqnarray*}
270\gamma_{n-j}+972\gamma_{n-j}^2+729\gamma_{n-j}^3&=&0,\\
972\gamma_{n-j}+729\gamma_{n-j}^2&=&0.
\end{eqnarray*}
Suppose that $\gamma_{n-j}$ does satisfy the above system.
If the characteristic is $2$, the first equation implies that $\gamma_{n-j}=0$, 
which is a contradiction.  Suppose then that the characteristic is not $2$.  
Solving the system, we have
$-3^{-2}(6+\sqrt{6})=\gamma_{n-j}=-3^{-1}4$, where $\sqrt 6$ may be any
square root of $6$.  
This leads to the conclusion 
that $30=0$.  Hence, the characteristic must be $5$.  
By Proposition~\ref{3deg}, we see that $j=n$ since 
$\gamma_{n-j}=-3^{-1}4\in\gf_q$.
However, this means that
$\gamma_{0}=2$, which is in contradiction with the 
choice of $\beta_0$ since $2$ is a perfect cube in this case.  
\end{proof}

\section{Comparison with Voloch's Work}\label{voloch_comparison}

The following is an improvement of a result of 
Voloch~\cite[\textsection 5]{voloch:2007}.
The proof is similar to the proof of the main theorem in~\cite{voloch:2007},
but more elementary in the sense that we avoid working with algebraic function 
fields.

\begin{thm}\label{voloch}
Let $q$ be a prime power, and let $0<\epsilon,\eta<1$. 
For $d$ sufficiently large, 
if $a\in\overline{\gf}_q$ has order $r$ and degree $d$ over $\gf_q$ with 
$r < d^{2-2\epsilon}$,  
then $a-1$ has order at least
$ \exp((1-\eta) \frac{2\epsilon}{3}d^{\epsilon/3} \log d )$.
The degree $d$ need only be large enough for the inequalities 
of~\eqref{1st d large requirement} and~\eqref{2nd d large requirement} to hold,
which depends only on the choices of $\epsilon$ and $\eta$.
\end{thm}
\begin{proof}
Let $0<\epsilon<1$ be given, and put 
$N := \ceil{d^{1-\epsilon}}$.
% Choose $a \in \overline{\gf}_q$ such that $\deg_{\gf_q} a = d$ and with 
% multiplicative
% order $r$. 
Note that $(r,q)=1$ since $r$ divides one less than a power of $q$ and $q$ is 
a prime power.  Also, note that the elements $a^{q^i}$, $0\le i \le d-1$, are distinct. 
It follows that the multiplicative order of $q$ modulo $r$ is exactly $d$.
For each coset $\Gamma$ of $\langle q\rangle$ in $(\zz/r\zz)^*$,
we define 
$J_\Gamma := \{ n\le N : n\mod r \in \Gamma\}.$
Note that there are $[(\zz/r\zz)^*:\langle q\rangle ] = \phi(r)/d$ 
cosets of $\langle q\rangle $ in $\left(\zz/r\zz\right)^*$.
Now
$$\sum_{\Gamma} |J_\Gamma| = \#\{ 1\le n \le N: \gcd(n,r) =1 \} 
=\frac{N \phi(r)}{r}+O(r^{\epsilon/10}),$$
where the sum is over all cosets of $\Gamma$ in $(\zz/r\zz)^*$.
Thus, there exists a coset $\Gamma = \gamma\langle q\rangle$ such that 
$|J_\Gamma|$ is at least the average.  That is,
$|J_\Gamma|\ge\frac{N d}{r}+O(dr^{\epsilon/10}/\phi(r))$.
Thus, there exists a positive constant $c_\epsilon$ so that
$|J_\Gamma|\ge\frac{Nd}{r}-c_\epsilon\frac{dr^{\epsilon/10}}{\phi(r)}
\ge d^\epsilon-c_\epsilon d^{\frac{\epsilon-\epsilon^2}{5}}$ since 
$d\le\phi(r)$.

Since $\gamma$ is coprime to $r$, write $\alpha\gamma+\beta r=1$ and 
take $c=a^\alpha$.  Then $a=c^\gamma$, and $c$ has order $r$ and degree 
at least $d$.
Let $b:=a-1$.
For each $n \in J_\Gamma$, there exists $j_n$ such that $n\equiv \gamma q^{j_n} \pmod r$.
Whence $c^n = c^{\gamma q^{j_n}} = a ^{q^{j_n}}$, and so 
$b^{q^{j_n}} = a^{q^{j_n}} - 1 = c^n-1$. 

Now, for every $I \subset J_\Gamma$ we write $b_I := \prod_{n\in I} (c^n-1) = 
\prod_{n_j \in I } b^{q^{n_j}}$ which is  a power of $b$.
Put  $T = \left[d^{\epsilon/3}\right]$, and observe that for $d$ sufficiently
large
\begin{equation}\label{1st d large requirement}
NT =\ceil{d^{1-\epsilon}}[d^{\epsilon/3}] < d.
\end{equation}
We claim that for all distinct $I,I' \subset J_{\Gamma}$ with $|I| = |I'| = T$
we have that $b_I \ne b_{I'}$.  Suppose that $b_I = b_{I'}$, and consider the non-zero
polynomial
$$p(t) = \prod_{n \in I} (t^n-1) -\prod_{n \in I'} (t^n-1).$$
Observe that $p(c) = b_I - b_{I'} = 0$, and so $\deg p(t) \ge \deg_{\gf_q} c \ge d$.
On the other hand, we have that $\deg p(t) \le NT < d$, a contradiction.
Thus $b_I \ne b_{I'}$ as claimed.

It follows that there are at least $\binom{|J_\Gamma|}{T}$ distinct powers of 
$b$.  Choose $0<\eta<1$.
Then, for $d$ sufficiently large, 
\begin{align}\label{2nd d large requirement}
\binom{|J_\Gamma|}{T} 
&\ge \left( \frac{|J_\Gamma|}{d^{\epsilon/3}}-1\right)^{d^{\epsilon/3}}
\ge\left(d^{2\epsilon/3}-c_\epsilon d^{-\frac{\epsilon(2+3\epsilon)}{15}}
-1\right)^{d^{\epsilon/3}}
\ge \exp\left((1-\eta) \frac{2\epsilon}{3}d^{\epsilon/3} \log d\right),
\end{align}
as required.
\end{proof}

In order to compare this result to Theorem~\ref{2tower}, one may 
choose $a=a_n$ to be a primitive $2^n$-th root of unity in $\overline{\gf}_q$.  
The degree of $a$ over $\gf_q$ will be $2^{n-\ord_2(q-1)}$.  
Then, for $n$ sufficiently large, the conditions of 
the above theorem will be satisfied. 
%Then the order of $a_n$ will be greater than 
%$\exp\left(\frac{2\epsilon\log 2}{3}n2^{\frac{n\epsilon}{3}}\right)$.
Similarly, one may choose $a$ to be a primitive 
$3^n$-th root of unity in $\overline{\gf}_q$ to compare with Theorem~\ref{3tower}.

Because of the requirement that $a$ must have low 
order relative to its degree, there are many fields in which 
Theorem~\ref{voloch} will not apply. 
Furthermore, one may check that even though the bound of Theorem~\ref{voloch} 
will eventually
dominate the bounds of Theorems~\ref{2tower} and~\ref{3tower}, 
there will always be a range (in terms of $n$) in which the
bounds of Theorems~\ref{2tower} and~\ref{3tower} will be larger.
For example, suppose we apply Theorem~\ref{voloch} to the case mentioned above,
and we maximize the bound of Theorem~\ref{voloch} by setting 
$\epsilon=1$ and $\eta=0$.  Further, suppose we minimize the bound of 
Theorem~\ref{2tower} 
by say assuming that $\ord_2(q-1)=1$.  Note that this will also serve to 
maximize the bound of Theorem~\ref{voloch}.  Under these assumptions, we 
may check that the bound of Theorem~\ref{2tower} will dominate for 
$n\le 11$.  However, we note that Theorem~\ref{voloch} does not actually apply 
if we choose $\epsilon=1$ and $\eta=0$; and the range of $n$ for which 
Theorem~\ref{2tower} will dominate will be larger for any appropriate
choice of $\epsilon$ and $\eta$.

\section{Examples of Theorems}\label{eg}

In this section we provide the data from the first several iterations for five examples 
of the main theorems: three for Theorem~\ref{2tower} and two for Theorem~\ref{3tower}.  
The tables in this section provide information 
about the orders of $\alpha_n$, $\beta_n$, $\delta_n$, and $\gamma_n$ 
in relation to our bound.  We have chosen to take logs of these numbers because 
of their size.  For each example, we note that the actual orders are much 
higher than our lower bounds.  Computations were aided by MAGMA~\cite{magma}. 

For our first example of Theorem~\ref{2tower}, we choose $q=5$ and $\alpha_0=2$.
\begin{table}[H]\caption{$q=5$; $\alpha_0=2$.}
\begin{tabular}{|c|c|c|c|c|}
\hline
 & & & & \\
$n$ & $\log_2\left|\gf_{5^{2^n}}^*\right|$ & $\log_2|\langle\alpha_n\rangle|$ & $\log_2|\langle\delta_n\rangle|$ & $\log_2\left(2^{\frac{1}{2}n^2+\frac{3}{2}n+1}\right)$\\
 & & & & \\
\hline
1 & 4.59 & 4.59 & 3.00 & 3.00\\ 
2 & 9.28 & 9.28 & 7.70 & 6.00\\ 
3 & 18.6 & 16.0 & 17.0 & 10.0\\ 
4 & 37.1 & 35.6 & 31.5 & 15.0\\ 
5 & 74.2 & 69.8 & 68.6 & 21.0\\ 
6 & 148. & 148. & 143. & 28.0\\ 
7 & 297. & 295. & 292. & 36.0\\ 
8 & 594. & 590. & 589. & 45.0\\ 
\hline
\end{tabular}
\end{table}

For our second example of Theorem~\ref{2tower}, we choose $q=9$ and
$\alpha_0=\zeta+2$, where $\zeta$ is a root of $x^2+1$.  Note that, in this 
example, $\delta_n$ is actually primitive for each of the first eight 
iterations.
\begin{table}[H]\caption{$q=9$; $\alpha_0=\zeta+2$.}
\begin{tabular}{|c|c|c|c|c|}
\hline
& & & & \\
$n$ & $\log_2\left|\gf_{9^{2^n}}^*\right|$ & $\log_2|\langle\alpha_n\rangle|$ &
$\log_2|\langle\delta_n\rangle|$ &
$\log_2\left(2^{\frac{1}{2}n^2+\frac{3}{2}n+3}\right)$\\ 
& & & & \\
\hline 
1 & 6.32 & 5.32 & 6.32 & 5.00\\ 
2 & 12.7 & 10.7 & 12.7 & 8.00\\ 
3 & 25.4 & 22.4 & 25.4 & 12.0\\ 
4 & 50.8 & 46.8 & 50.8 & 17.0\\ 
5 & 102. & 96.5 & 102. & 23.0\\ 
6 & 203. & 197. & 203. & 30.0\\ 
7 & 406. & 399. & 406. & 38.0\\ 
8 & 812. & 804. & 812. & 47.0\\ 
\hline
\end{tabular}
\end{table}

For our final example of Theorem~\ref{2tower}, we choose $q=121$ and
$\alpha_0=\eta^8$, where $\eta$ is a root of $x^2+7x+2$.  Here, $\delta_n$ 
is primitive except for $n=3$ and $n=7$.
\begin{table}[H]\caption{$q=121$; $\alpha_0=\eta^8$.}
\begin{tabular}{|c|c|c|c|c|}
\hline
& & & & \\
$n$ & $\log_2\left|\gf_{121^{2^n}}^*\right|$ & $\log_2|\langle\alpha_n\rangle|$ &
$\log_2|\langle\delta_n\rangle|$ &
$\log_2\left(2^{\frac{1}{2}n^2+\frac{3}{2}n+3}\right)$\\ 
& & & & \\
\hline 
1 & 13.8 & 11.8 & 13.8 & 5.00\\ 
2 & 27.7 & 26.7 & 27.7 & 8.00\\ 
3 & 55.4 & 50.8 & 53.0 & 12.0\\ 
4 & 111. & 109. & 111. & 17.0\\ 
5 & 222. & 216. & 222. & 23.0\\ 
6 & 443. & 440. & 443. & 30.0\\ 
7 & 886. & 874. & 883. & 38.0\\ 
\hline
\end{tabular}
\end{table}

For our first example of Theorem~\ref{3tower}, we choose $q=7$ and $\beta_0=3$.
In this example, $\gamma_n$ appears to alternate between being primitive and 
not.
\begin{table}[H]\caption{$q=7$; $\beta_0=3$.}
\begin{tabular}{|c|c|c|c|c|}
\hline
& & & & \\
$n$ & $\log_2\left|\gf_{7^{3^n}}^*\right|$ & $\log_2|\langle\beta_n\rangle|$ &
$\log_2|\langle\gamma_n\rangle|$ & 
$\log_2\left(3^{\frac{1}{2}n^2+\frac{3}{2}n+1}\right)$\\ 
& & & & \\
\hline
1 & 8.42 & 7.41 & 5.84 & 4.76\\ 
2 & 25.3 & 25.3 & 25.3 & 9.52\\ 
3 & 75.8 & 75.8 & 74.2 & 15.8\\ 
4 & 228. & 228. & 228. & 23.8\\ 
5 & 682. & 681. & 681. & 33.3\\ 
\hline
\end{tabular}
\end{table}

For our second example of Theorem~\ref{3tower}, we choose $q=16$ and 
$\beta_0=\xi$, where $\xi$ is a root of $x^4+x+1$.  Note that here $\gamma_n$ 
is primitive for each of the first five iterations.
\begin{table}[H]\caption{$q=16$; $\beta_0=\xi$.}
\begin{tabular}{|c|c|c|c|c|}
\hline
& & & & \\
$n$ & $\log_2\left|\gf_{16^{3^n}}^*\right|$ & $\log_2|\langle\beta_n\rangle|$ &
$\log_2|\langle\gamma_n\rangle|$ & 
$\log_2\left(3^{\frac{1}{2}n^2+\frac{3}{2}n+1}\right)$\\ 
& & & & \\
\hline 
1 & 12.0 & 8.83 & 12.0 & 4.76\\
2 & 36.0 & 31.2 & 36.0 & 9.52\\
3 & 108. & 102. & 108. & 15.8\\
4 & 324. & 316. & 324. & 23.8\\
5 & 972. & 962. & 972. & 33.3\\
\hline
\end{tabular}
\end{table}

\bibliography{references}
\end{document}